\documentclass[11pt]{amsart}
\usepackage{amsmath, amsfonts, amsthm, amssymb, multicol}
\usepackage{graphicx}
\usepackage{float}
\usepackage{verbatim}

\allowdisplaybreaks

\usepackage[square,sort,comma,numbers]{natbib}
\usepackage{fancyhdr}
 
\numberwithin{equation}{section}

\newtheorem{thm}{Theorem}[section]
\newtheorem{lma}[thm]{Lemma}

\renewcommand{\epsilon}{\varepsilon}

\newcommand{\conp}{\textup{con}^+(\mathbb{D}^n)}

\renewcommand{\geq}{\geqslant}
\renewcommand{\leq}{\leqslant}

\newcommand{\ubd}{\overline{\dim}_{\textup{B}}}
\newcommand{\hd}{\dim_{\textup{H}}}
\newcommand{\ad}{\dim_{\textup{A}}}

\title{The Poincar\'e exponent \\and the \\ dimensions of Kleinian limit sets}

\author{Jonathan M. Fraser\\ \\
 U\MakeLowercase{niversity of} S\MakeLowercase{t} A\MakeLowercase{ndrews}, S\MakeLowercase{cotland} \\
\MakeLowercase{Email: jmf32@st-andrews.ac.uk}}
\thanks{The  author was  financially supported by an \emph{EPSRC Standard Grant} (EP/R015104/1) and a  \emph{Leverhulme Trust Research Project Grant} (RPG-2019-034). He thanks Amlan Banaji and Liam Stuart for helpful comments.}

\begin{document}


\maketitle
\thispagestyle{empty}

\begin{abstract}
We provide a proof of the (well-known) result that the Poincar\'e exponent of a non-elementary Kleinian group is a lower bound for the upper box dimension of the limit set. Our proof only uses elementary hyperbolic and fractal geometry.
\end{abstract}

\section{Kleinian groups, limit sets, and the Poincar\'e exponent}

For integers $n \geq 2$, $n$-dimensional hyperbolic space can be modelled by the Poincar\'e ball 
\[
\mathbb{D}^n = \{ z \in \mathbb{R}^n: |z| <1\}
\]
equipped with the hyperbolic metric $d_H$ given by
\[
|ds| = \frac{2|dz|}{1-|z|^2}.
\]
The group of orientation preserving isometries of $(\mathbb{D}^n, d_H)$ is the group of conformal automorphisms of $\mathbb{D}^n$, which we denote by $\conp$.  A good way to get a handle on this group is to view it as the (orientation preserving) stabliser of  $\mathbb{D}^n$ as a subgroup of the M\"obius group acting on $\mathbb{R}^n \cup \{\infty\}$.  This group consists of maps given by the composition of reflections in spheres.  

A group $\Gamma \leq  \conp$ is called \emph{Kleinian} if it is discrete.  Kleinian groups generate beautiful fractal limit sets defined by
\[
L(\Gamma) = \overline{\Gamma(0)} \setminus \Gamma(0)
\]
where $\Gamma(0) = \{ g(0) : g \in \Gamma\}$ is the orbit of 0 under $\Gamma$ and the closure is the Euclidean closure. Discreteness of $\Gamma$ implies that all $\Gamma$-orbits are locally finite in $\mathbb{D}^n$ and this ensures that $L(\Gamma) \subseteq S^{n-1}$.  Here $S^{n-1}$ is the `boundary at infinity' of hyperbolic space.  A Kleinian group is called \emph{non-elementary} if its limit set contains at least 3 points, in which case it is necessarily an uncountable perfect set. 

The \emph{Poincar\'e exponent} captures the coarse rate of accumulation to the limit set and is defined as the exponent of convergence of the \emph{Poincar\'e series} 
\[
P_\Gamma(s) = \sum_{g \in \Gamma } \exp(-sd_H(0,g(0))) = \sum_{g \in \Gamma} \left(\frac{1-|g(0)|}{1+|g(0)|} \right)^s
\]
for $s \geq 0$.  The \emph{Poincar\'e exponent} is therefore
\[
\delta(\Gamma) = \inf\{ s \geq 0 : P_\Gamma(s) <\infty\}.
\]
It is a simple exercise to show that the \emph{Poincar\'e series}  may be defined using the orbit of an arbitrary $z \in \mathbb{D}^n$  at the expense of multiplicative constants depending only on $z$.  In particular, the exponent of convergence does not depend on the choice of $z$. (The definition above uses $z=0$.) For more background on hyperbolic geometry and Kleinian groups see \cite{beardon, maskit}.

There has been a great deal of interest in computing or estimating the fractal dimension of the limit set $L(\Gamma)$  (as a subset of Euclidean space $\mathbb{R}^n$) and the Poincar\'e exponent plays a central role.  We write $\hd, \, \ubd, \, \ad$ to denote the Hausdorff, upper box, and Assouad dimensions respectively.  These constitute three distinct and well-studied notions of fractal dimension. See \cite{falconer} for more background on dimension theory and fractal geometry, especially the box and Hausdorff dimensions, and \cite{book} for the Assouad dimension.  For all non-empty  bounded sets $F \subseteq \mathbb{R}^n$,
\[
0 \leq \hd F \leq \ubd F \leq \ad F \leq n.
\]
For all non-elementary Kleinian groups,  
\[
\delta(\Gamma) \leq \hd L(\Gamma)
\]
and for non-elementary \emph{geometrically finite}  Kleinian groups,  
\[
\delta(\Gamma) = \hd L(\Gamma) = \ubd L(\Gamma).
\]
See \cite{bowditch} for more details on geometric finiteness.  Roughly speaking it means that the Kleinian group admits a reasonable fundamental domain. The equality of Hausdorff dimension and Poincar\'e exponent in the geometrically finite setting  goes back to Sullivan \cite{sullivan}, see also Patterson  \cite{patterson}.  The coincidence with box dimension in this case was proved (rather later) independently by  Bishop and Jones \cite{bishopjones} and Stratmann and Urba\'nski \cite{su}.  The fact that the Poincar\'e exponent is always a lower bound for the Hausdorff dimension (without the assumption of geometric finiteness) is due to Bishop and Jones \cite{bishopjones}.  See the survey \cite{stratmann}.  In the presence of parabolic elements the Assouad dimension can be strictly greater than $\delta(\Gamma) $, even in the geometrically finite situation, see \cite{kleinian}. 

In the geometrically infinite setting, $\delta(\Gamma) < \hd L(\Gamma) < \ubd L(\Gamma)$ is possible, and it is an intriguing open problem to determine if $\hd L(\Gamma) = \ubd L(\Gamma)$ for all finitely generated $\Gamma$ for $n \geq 4$.  For $n = 3$, Bishop and Jones proved that for finitely generated, geometrically infinite $\Gamma$, $\hd L(\Gamma) = \ubd L(\Gamma)=2$, see \cite{bishopjones}.   This result was extended by Bishop to \emph{analytically finite} $\Gamma$ \cite{bishop, bishopinvent}. Falk and Matsuzaki characterised the upper box dimension of an arbitrary non-elementary Kleinian group in terms of the \emph{convex core entropy} \cite{falk}.  This can also be expressed as the exponent of convergence of an `extended Poincar\'e series', but is more complicated to introduce.  

  Proving the  general inequality $\delta(\Gamma) \leq \hd L(\Gamma)$  involves carefully constructing a measure supported on the limit set and applying the mass distribution principle.  Our investigation  begins with the following question: since (upper) box dimension is a simpler concept than Hausdorff dimension, can we prove the weaker inequality $\delta(\Gamma) \leq \ubd L(\Gamma)$ using only elementary methods?  We provide a short and self-contained proof of this estimate in the sections which follow.  It is instructive to think about why our  proof fails to prove the equality $\delta(\Gamma) = \ubd L(\Gamma)$ in general and what sort of extra assumptions on $\Gamma$ would be needed to  `upgrade' the proof to yield this stronger conclusion.

The (upper) box dimension of a non-empty bounded set $F \subseteq \mathbb{R}^n$ can be defined in terms of the asymptotic behaviour of the volume of the $r$-neighbourhood of $F$.  Given $r>0$ the $r$-neighbourhood of $F$ is denoted by $F_r$ and consists of all points in $\mathbb{R}^n$ which are at Euclidean distance less than or equal to $r$ from a point in $F$. Write $V_E$ to denote the Euclidean volume, that is, $n$-dimensional Lebesgue measure.  If $V_E(F) =0$, then $V_E(F_r) \to 0$ as $r \to 0$.  The upper box dimension   of $F$ captures this rate of decay and is defined formally by
\[
\ubd F = n-\liminf_{r \to 0} \frac{\log V_E(F_r) }{\log r}.
\] 
Another elementary proof of the estimate $\delta(\Gamma) \leq \ubd L(\Gamma)$, at least for $n=2,3$, can be found in \cite[Lemmas 2.1 and 3.1]{bishop}.  Here the connection is made via `Whitney squares'.

\section{Proof of dimension estimate}

Let $\Gamma$ be an arbitrary non-elementary Kleinian  group acting on the Poincar\'e ball and $\delta(\Gamma)$ denote the associated  Poincar\'e exponent.   We prove the following (well-known)  inequality:
\begin{equation} \label{main}
\delta(\Gamma) \leq \ubd L(\Gamma).
\end{equation}

Throughout we write $A \lesssim B$ to mean there is a constant $c >0$ such that $A \leq cB$.  Similarly, we write $A \gtrsim B$ if $B \lesssim A$ and $A \approx B$ if $A \lesssim B$ and $A \gtrsim B$.  The implicit constants may depend on $\Gamma$  and other fixed parameters, but it will be crucial that they never depend on the scale $r>0$ used to compute the box dimension or on a specific element $g \in \Gamma$.

\subsection{Elementary estimates from hyperbolic geometry}

 Since $\Gamma$ is non-elementary, it is easy to see that it must contain a loxodromic element, $h$. Loxodromic elements have precisely two fixed points on the boundary at infinity.  Let $z \in \mathbb{D}^{n}$ be a point lying on the (doubly infinite) geodesic ray joining the fixed points of $h$.  We may assume $z$ is not fixed by any elliptic elements in $\Gamma$ since it is an elementary fact that the set of elliptic fixed points is discrete.  Choose $a>0$ such that the set
\[
\{ B_{H}(g(z),a)\}_{g \in \Gamma}
\]
is pairwise disjoint, where $B_H(g(z),a)$ denotes the closed hyperbolic ball centred at $g(z)$ with radius $a$.  To see that such an $a$ exists  recall that  the orbit $\Gamma(z)$ is locally finite.  As such, $a$ can be chosen such that $B_{H}(z,2a)$ contains only one point from the orbit $\Gamma(z)$, namely $z$ itself.  Then the pairwise disjointness of the collection $\{ B_{H}(g(z),a)\}_{g \in \Gamma}$ is guaranteed since if $y \in B_{H}(g_1(z), a) \cap  B_{H}(g_2(z), a) $ for distinct $g_1,g_2 \in \Gamma$, then 
\[
d_{H}(z, g_1^{-1}g_2(z)) = d_{H}(g_1(z), g_2(z))  \leq d_{H}(g_1(z), y)+d_{H}(y, g_2(z)) \leq 2a
\]
which gives $g_1^{-1}g_2(z) \in B_{H}(z,2a)$, a contradiction.  

We will use  the simple volume estimate 
\begin{equation} \label{size}
V_E(B_{H}(g(z),a)) \approx  (1-|g(z)|)^n
\end{equation}
for  all $g \in \conp$, where the implicit constants  are independent of $g$ and $z$, but depend on $a$ and $n$.  This follows since  $B_{H}(g(z),a)$ is a \emph{Euclidean} ball with diameter  comparable to $1-|g(z)|$ (most likely not centred at $g(z)$).  To derive this explicitly it is useful to recall the (well-known and easily derived) formula for hyperbolic distance
\[
d_H(0,w) = \log \frac{1+|w|}{1-|w|}, \qquad (w \in \mathbb{D}^n).
\]
 The next result says that if $1-|g(z)|$ is small, then the image of a fixed set under $g$  must be contained in a comparably  small neighbourhood of the limit set.  This is the only point in the proof where the fact that the group is non-elementary is used.  It is instructive to find an example of an elementary group where the conclusion fails.  

\begin{lma} \label{volume}
Let $a,z$ be as above.  There exists a  constant $c>0$ depending only on $\Gamma$, $a$ and $z$ such that if $g \in \Gamma$ is such that $1-|g(z)|< 2^{-k+1}$ for a positive integer $k$, then
\[
B_{H}(g(z),a)  \subseteq L(\Gamma)_{c2^{-k}}.
\]
\end{lma}

\begin{proof}
The idea is that there must be a loxodromic fixed point close to $g(z)$ and loxodromic fixed points are necessarily in the limit set.  Indeed, $g(z)$ lies on the geodesic ray joining the  fixed points of the loxodromic map $g h g^{-1}$.  These fixed points are the images of the fixed points of $h$ under $g$ and at least one of them must lie in the smallest Euclidean sphere passing through $g(z)$ and intersecting the boundary $S^{n-1}$ at right angles. This uses the fact that geodesic rays are orthogonal to the boundary and $g$ is conformal. The diameter of this  sphere is
\[
\lesssim  1-|g(z)|  < 2^{-k+1}
\]
and the result follows, recalling that the Euclidean diameter of $B_{H}(g(z),a)$ is $\approx 1-|g(z)| $.
\end{proof}

\subsection{Estimating the Poinca\'re series using the limit set }

Let $s>t>\ubd L(\Gamma)$.  Then by definition
\begin{equation} \label{box}
V_E(L(\Gamma)_{r}) \lesssim r^{n-t}
\end{equation}
for all $0<r<c/2$ with implicit constant independent of $r$ but depending on $t$ and where $c$ is the constant from Lemma \ref{volume}. Then
\begin{eqnarray*}
P_\Gamma(s) &\approx& \sum_{g \in \Gamma} \left(\frac{1-|g(z)|}{1+|g(z)|} \right)^s \\ \\
&\approx& \sum_{k=1}^\infty \sum_{\substack{g \in \Gamma: \\2^{-k} \leq 1-|g(z)| < 2^{-k+1}}} (1-|g(z)|)^s \\ \\
 &\approx& \sum_{k=1}^\infty 2^{-k(s-n)} \sum_{\substack{g \in \Gamma: \\2^{-k} \leq 1-|g(z)| < 2^{-k+1}}} (1-|g(z)|)^n \\ \\
 &\lesssim& \sum_{k=1}^\infty 2^{-k(s-n)} \sum_{\substack{g \in \Gamma: \\2^{-k} \leq 1-|g(z)| < 2^{-k+1}}} V_E(B_{H}(g(z),a)) \qquad \text{(by \eqref{size})} \\ \\
 &\lesssim& \sum_{k=1}^\infty 2^{-k(s-n)}  V_E(L(\Gamma)_{c2^{-k}}) \qquad \text{(by Lemma \ref{volume} and choice of $a$)}\\ \\
 &\lesssim& \sum_{k=1}^\infty 2^{-k(s-n)} 2^{-k(n-t)} \qquad \text{(by \eqref{box})} \\ \\
 &=& \sum_{k=1}^\infty 2^{-k(s-t)}\\ \\
&<&\infty.
\end{eqnarray*}
Therefore  $\delta(\Gamma) \leq s$ and letting $s \to \ubd L(\Gamma)$ proves   \eqref{main}.

\end{document}